\documentclass[3p,times]{amsart}

\usepackage{amssymb}

\usepackage{amsthm}

\usepackage{enumerate} % for enumerate [(i)]/[(a)]/[(1)]

\usepackage{amsmath}
\usepackage{amssymb}

\newtheorem{theorem}{Theorem}[section] 
\newtheorem{lemma}[theorem]{Lemma}   
\newtheorem{corollary}[theorem]{Corollary}
\newtheorem{proposition}[theorem]{Proposition}

\newtheorem{main-theorem}{Theorem}

\newtheorem*{problem*}{Problem}

\theoremstyle{definition}

\newtheorem*{question*}{Question}

\renewcommand{\mod}{\operatorname{mod}}

\newcommand{\alg}{\operatorname{alg}}
\newcommand{\rad}{\operatorname{rad}}
\newcommand{\soc}{\operatorname{soc}}

\newcommand{\umod}{\operatorname{\underline{mod}}}

\newcommand{\Ext}{\operatorname{Ext}}
\newcommand{\Hom}{\operatorname{Hom}}
\newcommand{\Ker}{\operatorname{Ker}}
\newcommand{\T}{\operatorname{T}}

\newcommand{\GL}{\operatorname{GL}}
\newcommand{\op}{\operatorname{op}}

\newcommand{\bN}{\mathbb{N}}
\newcommand{\bP}{\mathbb{P}}

\newcommand{\cB}{\mathcal{B}}

\newcommand{\overbar}[1]{\mkern 5mu\overline{\mkern-5mu#1\mkern-5mu}\mkern 5mu}

\usepackage{etex} % needed because of using both tikz and xymatrix below
\usepackage{dsfont}         % for \mathds{1}
\usepackage[h]{esvect}

\usepackage[matrix,arrow,curve,frame,arc]{xy}

\usepackage{pgf}
\usepackage{tikz}
\usepackage{xcolor}

\usetikzlibrary{arrows,shapes,automata,backgrounds,petri,calc}

\newcommand{\tikzAngleOfLine}{\tikz@AngleOfLine}
\def\tikz@AngleOfLine(#1)(#2)#3{%
\pgfmathanglebetweenpoints{%
\pgfpointanchor{#1}{center}}{%
\pgfpointanchor{#2}{center}}
\pgfmathsetmacro{#3}{\pgfmathresult}%
}

\begin{document}

\title{Higher tetrahedral algebras}

{\def\thefootnote{}
\footnote{The research was supported by the research grant
DEC-2011/02/A/ST1/00216 of the National Science Center Poland.}
}

\author[K. Edrmann]{Karin Erdmann}
\address[Karin Erdmann]{Mathematical Institute,
   Oxford University,
   ROQ, Oxford OX2 6GG,
   United Kingdom}
\email{erdmann@maths.ox.ac.uk}

\author[A. Skowro\'nski]{Andrzej Skowro\'nski}
\address[Andrzej Skowro\'nski]{Faculty of Mathematics and Computer Science,
   Nicolaus Copernicus University,
   Chopina~12/18,
   87-100 Toru\'n,
   Poland}
\email{skowron@mat.uni.torun.pl}

\begin{abstract}
We introduce and study the higher tetrahedral algebras,
an exotic family of finite-dimensional tame
symmetric algebras over an algebraically closed field.
The Gabriel quiver of such an algebra is the triangulation 
quiver associated to the coherent orientation of the tetrahedron.
Surprisingly, these algebras occurred  in the classification
of all algebras of generalized quaternion type, but are not
weighted surface algebras.
We prove that a higher tetrahedral algebra is periodic
if and only if it is non-singular.

\bigskip

 

\end{abstract}

\maketitle

\section{Introduction and the main results}\label{sec:intro}

%\noindent
Throughout this paper, $K$ will denote a fixed algebraically closed field.
By an algebra we mean an associative finite-dimensional $K$-algebra
with an identity.
For an algebra $A$, we denote by $\mod A$ the category of
finite-dimensional right $A$-modules and by $D$ the standard
duality $\Hom_K(-,K)$ on $\mod A$.
An algebra $A$ is called \emph{self-injective}
if $A_A$ is injective in $\mod A$, or equivalently,
the projective modules in $\mod A$ are injective.
A prominent class of self-injective algebras is formed
by the \emph{symmetric algebras} $A$ for which there exists
an associative, non-degenerate symmetric $K$-bilinear form
$(-,-): A \times A \to K$.
Classical examples of symmetric algebras are provided
by the blocks of group algebras of finite groups and
the Hecke algebras of finite Coxeter groups.
In fact, any algebra $A$ is the quotient algebra
of its trivial extension algebra $\T(A) = A \ltimes D(A)$,
which is a symmetric algebra.

From the remarkable Tame and Wild Theorem of Drozd
(see \cite{CB1,Dr})
the class of algebras over $K$ may be divided
into two disjoint classes.
The first class consists of the \emph{tame algebras}
for which the indecomposable modules occur in each dimension $d$
in a finite number of discrete and a finite number of one-parameter
families.
The second class is formed by the \emph{wild algebras}
whose representation theory comprises
the representation theories of all
algebras over $K$.
Accordingly, we may realistically hope to classify
the indecomposable finite-dimensional modules only
for the tame algebras.
Among the tame algebras we may distinguish 
the \emph{algebras of polynomial growth}
for which the number of one-parameter families 
of indecomposable  modules in each dimension $d$
is bounded by $d^m$ for some positive integer $m$
(depending only on the algebra)
whose representation theory is usually well understood
(see \cite{BES4,Sk1,Sk2} for some general results).
On the other hand, the representation theory of 
tame algebras of non-polynomial growth is still only emerging.

Let $A$ be an algebra.
Given a module $M$ in  $\mod A$, its \emph{syzygy}
is defined to be the kernel $\Omega_A(M)$ of a minimal
projective cover of $M$ in $\mod A$.
The syzygy operator $\Omega_A$ is a very important tool
to construct modules in $\mod A$ and relate them.
For $A$ self-injective, it induces an equivalence
of the stable module category $\umod A$,
and its inverse is the shift of a triangulated structure
on $\umod A$ \cite{Ha1}.
A module $M$ in $\mod A$ is said to be \emph{periodic}
if $\Omega_A^n(M) \cong M$ for some $n \geq 1$, and if so
the minimal such $n$ is called the \emph{period} of $M$.
The action of $\Omega_A$ on $\mod A$ can effect
the algebra structure of $A$.
For example, if all simple modules in $\mod A$ are periodic,
then $A$ is a self-injective algebra.
An algebra $A$ is defined to be \emph{periodic} if it is periodic
viewed as a module over the enveloping algebra
$A^e = A^{\op} \otimes_K A$, or equivalently,
as an $A$-$A$-bimodule.
It is known that if $A$ is a periodic algebra of period $n$ then
for any indecomposable non-projective module $M$
in $\mod A$  the syzygy $\Omega_A^n(M)$ is isomorphic to $M$.

Finding or possibly classifying periodic algebras
is an important problem. It is very interesting because of
connections
with group theory, topology, singularity theory
and cluster algebras.
Periodicity of an algebra, and its period,
are invariant under derived equivalences 
\cite{Ric2} (see also \cite{ESk3}).
Therefore, to study periodic algebras 
we may assume that the algebras are basic
and indecomposable.

We are concerned with the classification of all periodic tame
symmetric algebras.
In \cite{Du1} Dugas proved that every representation-finite
self-injective algebra, without simple blocks, is a periodic algebra.
We note that, by general theory (see \cite[Section~3]{Sk2}),
a basic, indecomposable, non-simple, symmetric algebra $A$
is representation-finite if and only if $A$ is socle equivalent
to an algebra $\T(B)^G$ of invariants of the trivial extension
algebra $\T(B)$ of a tilted algebra $B$ of Dynkin type
with respect to free action of a finite cyclic group $G$.
The representation-infinite, indecomposable,
periodic algebras of polynomial growth were classified
by Bia\l kowski, Erdmann and Skowro\'nski in \cite{BES4}
(see also \cite{Sk1,Sk2}).
In particular, it follows from \cite{BES4}
that every basic, indecomposable,
representation-infinite symmetric tame algebra
of polynomial growth is socle
equivalent to an algebra $\T(B)^G$
of invariants of the trivial extension algebra $\T(B)$
of a tubular algebra $B$ of tubular type
$(2,2,2,2)$, $(3,3,3)$, $(2,4,4)$, $(2,3,6)$
(introduced by Ringel \cite{R}) with respect to free
action of a finite cyclic group $G$.

Recently we introduced in \cite{ESk-WSA}
the weighted surface algebras of triangulated surfaces
with arbitrary oriented triangles and proved that all these algebras,
except the singular tetrahedral algebras, are periodic tame
symmetric algebras of period $4$.
Here, we investigate the periodicity of higher tetrahedral algebras,
being ``higher analogues'' of the tetrahedral algebras
studied in \cite{ESk-WSA}.

Consider the tetrahedron 
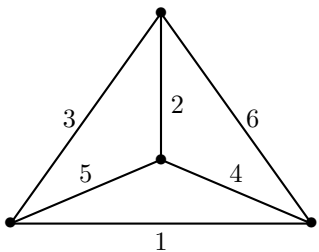
\[
\begin{tikzpicture}
[scale=1]
\node (A) at (-2,0) {$\bullet$};
\node (B) at (2,0) {$\bullet$};
\node (C) at (0,.85) {$\bullet$};
\node (D) at (0,2.8) {$\bullet$};
\coordinate (A) at (-2,0) ;
\coordinate (B) at (2,0) ;
\coordinate (C) at (0,.85) ;
\coordinate (D) at (0,2.8) ;
\draw[thick]
(A) edge node [left] {3} (D)
(D) edge node [right] {6} (B)
(D) edge node [below right] {2} (C)
(A) edge node [above] {5} (C)
(C) edge node [above] {4} (B)
(A) edge node [below] {1} (B) ;
\end{tikzpicture}
\]
with the coherent orientation
of triangles:
$(1\ 5\ 4)$, $(2\ 5\ 3)$, $(2\ 6\ 4)$, $(1\ 6\ 3)$.
Then, following \cite{ESk-WSA}, we have the associated 
triangulation quiver $(Q,f)$ of the form
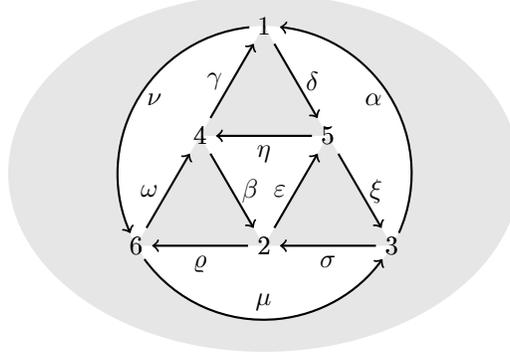
\[
\begin{tikzpicture}
[scale=.85]
\node (1) at (0,1.72) {$1$};
\node (2) at (0,-1.72) {$2$};
\node (3) at (2,-1.72) {$3$};
\node (4) at (-1,0) {$4$};
\node (5) at (1,0) {$5$};
\node (6) at (-2,-1.72) {$6$};
\coordinate (1) at (0,1.72);
\coordinate (2) at (0,-1.72);
\coordinate (3) at (2,-1.72);
\coordinate (4) at (-1,0);
\coordinate (5) at (1,0);
\coordinate (6) at (-2,-1.72);
\fill[fill=gray!20]
      (0,2.22cm) arc [start angle=90, delta angle=-360, x radius=4cm, y radius=2.8cm]
 --  (0,1.72cm) arc [start angle=90, delta angle=360, radius=2.3cm]
     -- cycle;
\fill[fill=gray!20]
    (1) -- (4) -- (5) -- cycle;
\fill[fill=gray!20]
    (2) -- (4) -- (6) -- cycle;
\fill[fill=gray!20]
    (2) -- (3) -- (5) -- cycle;

\node (1) at (0,1.72) {$1$};
\node (2) at (0,-1.72) {$2$};
\node (3) at (2,-1.72) {$3$};
\node (4) at (-1,0) {$4$};
\node (5) at (1,0) {$5$};
\node (6) at (-2,-1.72) {$6$};
\draw[->,thick] (-.23,1.7) arc [start angle=96, delta angle=108, radius=2.3cm] node[midway,right] {$\nu$};
\draw[->,thick] (-1.87,-1.93) arc [start angle=-144, delta angle=108, radius=2.3cm] node[midway,above] {$\mu$};
\draw[->,thick] (2.11,-1.52) arc [start angle=-24, delta angle=108, radius=2.3cm] node[midway,left] {$\alpha$};
\draw[->,thick]
(1) edge node [right] {$\delta$} (5)
(2) edge node [left] {$\varepsilon$} (5)
(2) edge node [below] {$\varrho$} (6)
(3) edge node [below] {$\sigma$} (2)
(4) edge node [left] {$\gamma$} (1)
(4) edge node [right] {$\beta$} (2)
(5) edge node [right] {$\xi$} (3)
(5) edge node [below] {$\eta$} (4)
(6) edge node [left] {$\omega$} (4)
;
\end{tikzpicture}
\]
where 
$f$ is the permutation of arrows of order $3$ described by
the shaded subquivers.
We denote by $g$ the permutation on the set of arrows of $Q$ 
whose $g$-orbits are the four white $3$-cycles.

Let $m \geq 2$ be a natural number and $\lambda \in K$.
We denote by $\Lambda(m,\lambda)$ the algebra given
by the above quiver and the relations:
\begin{align*}
 \gamma\delta &= \beta\varepsilon + \lambda (\beta\varrho\omega)^{m-1} \beta\varepsilon, 
 &
 \delta\eta &= \nu\omega,
 &
 \eta\gamma &= \xi\alpha,
 & 
 \nu \mu &= \delta\xi ,
\\
 \varrho\omega &= \varepsilon\eta + \lambda (\varepsilon\xi\sigma)^{m-1} \varepsilon\eta,
 &
 \omega\beta &= \mu\sigma,
 &
 \beta\varrho &= \gamma\nu ,
 &
 \mu \alpha &= \omega \gamma ,
\\
 \xi\sigma &= \eta\beta + \lambda (\eta\gamma\delta)^{m-1} \eta\beta,
 &
 \sigma\varepsilon &= \alpha\delta,
 &
 \varepsilon\xi &= \varrho\mu,
 &
 \alpha\nu &= \sigma\varrho, 
\\
\omit\rlap{\qquad\quad$\big(\theta f(\theta) f^2(\theta)\big)^{m-1} \theta f(\theta) g\big(f(\theta)\big) = 0$ for any arrow $\theta$ in $Q$.}
\end{align*}

We call $\Lambda(m,\lambda)$
a \emph{higher tetrahedral algebra}.
Moreover, an algebra $\Lambda(m,\lambda)$ with
$\lambda \in K^* = K \setminus \{0\}$ is said to be
a \emph{non-singular higher tetrahedral algebra}.

The following two theorems describe some properties
of  higher tetrahedral  algebras.

\begin{main-theorem}
\label{th:main1}
Let $\Lambda = \Lambda(m,\lambda)$
be a higher tetrahedral algebra.
Then $\Lambda$ is a finite-dimensional symmetric algebra 
  with $\dim_K \Lambda = 36 m$.
\end{main-theorem}

\begin{main-theorem}
\label{th:main2}
Let $\Lambda = \Lambda(m,\lambda)$
be a higher tetrahedral algebra.
Then $\Lambda$ is a tame algebra of non-polynomial growth.
\end{main-theorem}

The following theorem is the main result of the paper.

\begin{main-theorem}
\label{th:main3}
Let $\Lambda = \Lambda(m,\lambda)$
be a higher tetrahedral algebra.
Then the following statements are equivalent:
\begin{enumerate}[(i)]
 \item
  $\mod \Lambda$ admits a periodic simple module.
 \item
  All simple modules in $\mod \Lambda$ are periodic of period $4$.
 \item
  $\Lambda$ is a periodic algebra of period $4$.
 \item
  $\Lambda$ is non-singular.
\end{enumerate}
\end{main-theorem}

Following \cite{ESk-AGQT},
an algebra $A$ is called an algebra of generalized quaternion
type if $A$ is representation-infinite tame symmetric 
and every simple module in $\mod A$ is periodic of period $4$.
We prove in \cite{ESk-AGQT} that an algebra $A$ is of generalized
quaternion type with $2$-regular Gabriel quiver if and only if 
$A$ is a socle deformation of a weighted surface algebra,
different from the singular tetrahedral algebra,
or is a non-singular higher tetrahedral algebra.

This paper is organized as follows.
In Section~\ref{sec:pre} we recall background 
on special biserial algebras and degenerations of algebras.
In Section~\ref{sec:bimodule} we describe our general approach and
results for constructing a minimal projective bimodule
resolution of an algebra with periodic simple modules.
Section~\ref{sec:proof1} is devoted to basic properties
of the higher  tetrahedral algebras and the proof of 
Theorem~\ref{th:main1}.
Sections \ref{sec:proof2} and \ref{sec:proof3} contain 
the proofs of 
Theorems \ref{th:main2} and \ref{th:main3}, respectively.

For general background on the relevant representation theory
we refer to the books
\cite{ASS,SS,SY}.

\section{Preliminary results}\label{sec:pre}

%\noindent
A \emph{quiver} is a quadruple $Q = (Q_0, Q_1, s, t)$
consisting of a finite set $Q_0$ of vertices,
a finite set $Q_1$ of arrows,
and two maps $s,t : Q_1 \to Q_0$ which associate
to each arrow $\alpha \in Q_1$ its source $s(\alpha) \in Q_0$
and  its target $t(\alpha) \in Q_0$.
We denote by $K Q$ the path algebra of $Q$ over $K$
whose underlying $K$-vector space has as its basis
the set of all paths in $Q$ of length $\geq 0$, and
by $R_Q$ the arrow ideal of $K Q$ generated by all paths 
in $Q$ of length $\geq 1$.
An ideal $I$ in $K Q$ is said to be admissible
if there exists $m \geq 2$ such that
$R_Q^m \subseteq I \subseteq R_Q^2$.
If $I$ is an admissible ideal in $K Q$, then
the quotient algebra $K Q/I$ is called
a bound quiver algebra, and is a finite-dimensional
basic $K$-algebra.
Moreover, $K Q/I$ is indecomposable if and only if
$Q$ is connected.
Every basic, indecomposable, finite-dimensional
$K$-algebra $A$ has a bound quiver presentation
$A \cong K Q/I$, where $Q = Q_A$ is the \emph{Gabriel
quiver} of $A$ and $I$ is an admissible ideal in $K Q$.
For a bound quiver algebra $A = KQ/I$, we denote by $e_i$,
$i \in Q_0$, the associated complete set of pairwise
orthogonal primitive idempotents of $A$, and by
$S_i = e_i A/e_i \rad A$ (respectively, $P_i = e_i A$),
$i \in Q_0$, the associated complete family of pairwise
non-isomorphic simple modules (respectively, indecomposable
projective modules) in $\mod A$.

Following \cite{SW}, an algebra $A$ is said to be
\emph{special biserial} if $A$ is isomorphic
to a bound quiver algebra $K Q/I$, where the bound
quiver $(Q,I)$ satisfies the following conditions:
\begin{enumerate}[(a)]
% \item[(a)]
 \item
  each vertex of $Q$ is a source and target of at most two arrows,
% \item[(b)]
 \item
  for any arrow $\alpha$ in $Q$ there are at most
  one arrow $\beta$ and at most one arrow $\gamma$
  with $\alpha \beta \notin I$ and $\gamma \alpha \notin I$.
\end{enumerate}
Moreover, if in addition $I$ is generated by paths of $Q$, then
$A = K Q/I$ is said to be a \emph{string algebra} \cite{BR}.
It was  proved in \cite{PS} that the class
of special biserial algebras coincides with the class
of biserial algebras (indecomposable projective modules
have biserial structure) which admit simply connected
Galois coverings.
Furthermore, by \cite[Theorem~1.4]{WW}
we know that  every special biserial agebra is a quotient algebra
of a symmetric special biserial algebra.
We also mention that, if $A$ is a self-injective
special biserial algebra, then $A/\soc(A)$ is a  string algebra.

The following  has been proved by Wald and Waschb\"usch
in \cite{WW} (see also \cite{BR,DS} for alternative
proofs).

\begin{proposition}
\label{prop:2.1}
Every special biserial algebra is tame.
\end{proposition}

For a positive integer $d$, we denote by $\alg_d(K)$ the affine
variety of associative $K$-algebra structures with identity on
the affine space $K^d$.
Then the general linear group $\GL_d(K)$ acts on $\alg_d(K)$
by transport of the structures, and the $\GL_d(K)$-orbits in
$\alg_d(K)$ correspond to the isomorphism classes of $d$-dimensional
algebras (see \cite{Kr} for details). We identify a $d$-dimensional
algebra $A$ with the point of $\alg_d(K)$ corresponding to it.
For two $d$-dimensional algebras $A$ and $B$, we say that $B$
is a \emph{degeneration} of $A$ ($A$ is a \emph{deformation} of $B$)
if $B$ belongs to the closure of the $\GL_d(K)$-orbit
of $A$ in the Zariski topology of $\alg_d(K)$.

Geiss' Theorem \cite{Ge} shows that if $A$ and $B$ are two
$d$-dimensional algebras, $A$ degenerates to $B$ and $B$ is a tame
algebra, then $A$ is also a tame algebra (see also \cite{CB2}).
We will apply this theorem in the following special situation.

\begin{proposition}
\label{prop:2.2}
Let $d$ be a positive integer, and $A(t)$, $t \in K$,
be an algebraic family in $\alg_d(K)$ such that $A(t) \cong A(1)$
for all $t \in K \setminus \{0\}$.
Then $A(1)$ degenerates to $A(0)$.
In particular, if $A(0)$ is tame, then $A(1)$ is tame.
\end{proposition}

A family of algebras $A(t)$, $t \in K$, in $\alg_d(K)$
is said to be \emph{algebraic} if the induced map
$A(-) : K \to \alg_d(K)$ is a regular map of affine varieties.

\section{Bimodule resolutions of self-injective algebras}\label{sec:bimodule}

%\noindent
In this section we describe a general approach
for proving that an algebra $A$ with periodic simple modules
is a periodic algebra.

Let $A = K Q/I$ be a bound quiver algebra,
and $e_i$, $i \in Q_0$, be the primitive idempotents of $A$
associated to the vertices of $Q$.
Then $e_i \otimes e_j$, $i,j \in Q_0$, form a set
of pairwise orthogonal primitive idempotents of
the enveloping algebra $A^e = A^{\op} \otimes_K A$
whose sum is the identity of $A^e$.
Hence,
$P(i,j) = (e_i \otimes e_j) A^e = A e_i \otimes e_j A$,
for $i,j \in Q_0$, form a complete set of pairwise non-isomorphic
indecomposable projective modules in $\mod A^e$
(see \cite[Proposition~IV.11.3]{SY}).

The following result by Happel \cite[Lemma~1.5]{Ha2} describes
the terms of a minimal projective resolution of $A$ in $\mod A^e$.

\begin{proposition}
\label{prop:3.1}
Let $A = K Q/I$ be a bound quiver algebra.
Then there is in $\mod A^e$ a minimal projective resolution of $A$
of the form
\[
  \cdots \rightarrow
  \bP_n \xrightarrow{d_n}
  \bP_{n-1} \xrightarrow{ }
  \cdots \rightarrow
  \bP_1 \xrightarrow{d_1}
  \bP_0 \xrightarrow{d_0}
  A \rightarrow 0,
\]
where
\[
  \bP_n = \bigoplus_{i,j \in Q_0}
          P(i,j)^{\dim_K \Ext_A^n(S_i,S_j)}
\]
for any $n \in \bN$.
\end{proposition}

The syzygy modules have an important property, a proof for the next
Lemma may be found in \cite[Lemma~IV.11.16]{SY}.

\begin{lemma}
\label{lem:3.2}
Let $A$ be an algebra.
For any positive integer $n$, the module  $\Omega_{A^e}^n(A)$
is  projective as a left $A$-module and also as a right
$A$-module.
\end{lemma}

There is no general recipe for
the differentials $d_n$ in Proposition~\ref{prop:3.1}, except for the first three which we will now describe.
We have
\[
  \bP_0 = \bigoplus_{i \in Q_0} P(i,i)
        = \bigoplus_{i \in Q_0} A e_i \otimes e_i A .
\]
The homomorphism $d_0 : \bP_0 \to A$ in $\mod A^e$ defined by
$d_0 (e_i \otimes e_i) = e_i$ for all $i \in Q_0$
is a minimal projective cover of $A$ in $\mod A^e$.
Recall that, for two vertices $i$ and $j$ in $Q$,
the number of arrows from $i$ to $j$ in $Q$ is equal
to $\dim_K \Ext_A^1(S_i,S_j)$
(see \cite[Lemma~III.2.12]{ASS}).
Hence we have
\[
  \bP_1 = \bigoplus_{\alpha \in Q_1} P\big(s(\alpha),t(\alpha)\big)
        = \bigoplus_{\alpha \in Q_1} A e_{s(\alpha)} \otimes e_{t(\alpha)} A
        .
\]
Then we have the following known fact (see \cite[Lemma~3.3]{BES4}
for a proof).

\begin{lemma}
\label{lem:3.3}
Let $A = K Q/I$ be a bound quiver algebra, and
$d_1 : \bP_1 \to \bP_0$ the homomorphism in $\mod A^e$
defined by
\[
 d_1(e_{s(\alpha)} \otimes e_{t(\alpha)}) =
   \alpha \otimes e_{t(\alpha)} - e_{s(\alpha)} \otimes \alpha
\]
for any arrow $\alpha$ in $Q$.
Then $d_1$ induces a minimal projective cover
$d_1 : \bP_1 \to \Omega_{A^e}^1(A)$ of
$\Omega_{A^e}^1(A) = \Ker d_0$ in $\mod A^e$.
In particular, we have
$\Omega_{A^e}^2(A) \cong \Ker d_1$ in $\mod A^e$.
\end{lemma}

We will denote  the  homomorphism
$d_1 : \bP_1 \to \bP_0$ by $d$.
For the algebras $A$ we will consider, the kernel
$\Omega_{A^e}^2(A)$ of $d$ will be generated,
as an $A$-$A$-bimodule, by some elements of $\bP_1$
associated to a set of relations generating the
admissible ideal $I$.
Recall that a relation in the path algebra $KQ$
is an element of the form
\[
  \mu = \sum_{r=1}^n c_r \mu_r
  ,
\]
where $c_1, \dots, c_r$ are non-zero elements of $K$ and
$\mu_r = \alpha_1^{(r)} \alpha_2^{(r)} \dots \alpha_{m_r}^{(r)}$
are paths in $Q$ of length $m_r \geq 2$, $r \in \{1,\dots,n\}$,
having a common source and a common target.
The admissible ideal $I$ can be generated by a finite set
of relations in $K Q$ (see \cite[Corollary~II.2.9]{ASS}).
In particular, the bound quiver algebra $A = K Q/I$ is given
by the path algebra $K Q$ and a finite number of identities
$\sum_{r=1}^n c_r \mu_r = 0$ given by a finite set of generators of
the ideal $I$.
Consider the $K$-linear homomorphism $\pi : K Q \to \bP_1$
which assigns to a path $\alpha_1 \alpha_2 \dots \alpha_m$ in $Q$
the element
\[
  \pi(\alpha_1 \alpha_2 \dots \alpha_m)
   = \sum_{k=1}^m \alpha_1 \alpha_2 \dots \alpha_{k-1}
                  \otimes \alpha_{k+1} \dots \alpha_m
\]
in $\bP_1$, where $\alpha_0 = e_{s(\alpha_1)}$
and $\alpha_{m+1} = e_{t(\alpha_m)}$.
Observe that
$\pi(\alpha_1 \alpha_2 \dots \alpha_m) \in e_{s(\alpha_1)} \bP_1 e_{t(\alpha_m)}$.
Then, for a relation $\mu = \sum_{r=1}^n c_r \mu_r$
in $K Q$ lying in $I$, we have an element
\[
  \pi(\mu) = \sum_{r=1}^n c_r \pi(\mu_r) \in e_i \bP_1 e_j ,
\]
where $i$ is the common source and $j$ is the common
target of the paths $\mu_1,\dots,\mu_r$.
The following lemma shows that relations always
produce elements in the kernel of $d_1$;  the proof 
is straightforward.

\begin{lemma}
\label{lem:3.4}
Let $A = K Q/I$ be a bound quiver algebra and
$d_1 : \bP_1 \to \bP_0$ the homomorphism in $\mod A^e$
defined in Lemma~\ref{lem:3.3}. 
Then for any relation $\mu$ in $K Q$ lying in $I$,
we have $d_1(\pi(\mu)) = 0$.
\end{lemma}

For an algebra $A = K Q/I$ in our context, we will see that
there exists a family of relations $\mu^{(1)},\dots,\mu^{(q)}$
generating the ideal $I$ such that the associated elements
$\pi(\mu^{(1)}), \dots, \pi(\mu^{(q)})$ generate
the $A$-$A$-bimodule $\Omega_{A^e}^2(A) = \Ker d_1$.
In fact, using Lemma~\ref{lem:3.2},
we will be able to show that
\[
  \bP_2 = \bigoplus_{j = 1}^q P\big(s(\mu^{(j)}),t(\mu^{(j)})\big)
        = \bigoplus_{j = 1}^q A e_{s(\mu^{(j)})} \otimes e_{t(\mu^{(j)})} A
        ,
\]
and the homomorphism $d_2 : \bP_2 \to \bP_1$ in $\mod A^e$ such that
\[
  d_2 \big(e_{s(\mu^{(j)})} \otimes e_{t(\mu^{(j)})}\big) = \pi(\mu^{(j)})
  ,
\]
for $j \in \{1,\dots,q\}$, defines a projective cover
of $\Omega_{A^e}^2(A)$ in $\mod A^e$.
In particular, we have $\Omega_{A^e}^3(A) \cong \Ker d_2$ in $\mod A^e$.
We will denote  this homomorphism $d_2$ by $R$.

For the next map $d_3 : \bP_3 \to \bP_2$, which we will call $S:=d_3$
later,
we do not have a general recipe. To define it, 
we need a set of minimal generators for 
$\Omega_{A^e}^3(A)$, and 
Proposition~\ref{prop:3.1}  tells us where we should look for them.

\section{Proof of Theorem~\ref{th:main1}}\label{sec:proof1}

Let $\Lambda = \Lambda(m,\lambda)$ for some $m \geq 2$ 
and $\lambda \in K$.
In this section we will study algebra properties of $\Lambda$, and in particular prove Theorem~\ref{th:main1}. 
The first results will be used to reduce calculations, and should also be of independent interest.

In order to construct a basis of $\Lambda$ with  good properties, we analyze the images of paths in $\Lambda$, they have very unusual properties.
We introduce some notation.
It follows from the relations defining $\Lambda$ that we may
define the elements
\begin{align*}
%&&
 {X}_1 &= \delta\eta\gamma = \nu\mu\alpha, 
& 
 {X}_2 &= \varrho\omega\beta = \varepsilon\xi\sigma, 
& 
 {X}_3 &= \alpha\nu\mu = \sigma\varepsilon\xi, 
%&& 
\\
%&&
 {X}_4 &= \gamma\delta\eta = \beta\varrho\omega, 
& 
 {X}_5 &= \eta\gamma\delta = \xi\sigma\varepsilon, 
& 
 {X}_6 &= \omega\beta\varrho = \mu\alpha\nu, 
%&& 
\end{align*}
given by products of the arrows around the shaded triangles.
Moreover, we define the elements
\begin{align*}
%&&
 \tilde{X}_2 &= \varepsilon\eta\beta, 
& 
 \tilde{X}_4 &= \beta\varepsilon\eta, 
& 
 \tilde{X}_5 &= \eta\beta\varepsilon. 
%&& 
\end{align*}
The quiver $Q$ of $\Lambda$ 
has an automorphism $\varphi$ of order $3$, defined as follows.
Its action on vertices is given by the cycles
\[
 (5 \ 4 \ 2)(1 \ 6 \ 3)
\]
and the action on arrows is
\[
 (\delta \ \omega \  \sigma) (\eta \ \beta \ \varepsilon) (\gamma \ \rho \ \xi) (\nu \ \mu \ \alpha).
\]

\begin{lemma}
\label{lem:4.1} 
The action of $\varphi$ extends to an algebra automorphism of $\Lambda$.
\end{lemma}

\begin{proof} 
We extend $\varphi$ to an algebra map of $K Q$. Then we must check that $\varphi$ preserves the relations, which is direct calculation. For example, 
\[
\varphi(\gamma\delta) = \varphi(\gamma)\varphi(\delta) = \rho\omega \ \ \mbox{ and } 
\  \varphi(\beta\varepsilon) = \varphi(\beta)\varphi(\varepsilon) = \varepsilon\eta,
\]
and 
\[
   \varphi(X_4) = \varphi(\gamma)\varphi(\delta)\varphi(\eta) = \rho\omega\beta = X_2.
\]
Hence,
$\varphi$ takes the relation 
for $\gamma\delta$ to the relation for $\rho\omega$. 
\end{proof}

\begin{lemma}
\label{lem:4.2} 
For each vertex $i$ of $Q$, 
the element $X_i^m$ belongs to the right socle of $\Lambda$.  
\end{lemma}

\begin{proof} 
It follows from the relations that, 
for each arrow $\theta$ in $Q$,
we have $X_{s(\theta)}^m \theta = 0$.  
For example, we have 
\[ 
  X_1^m\delta = (\nu\mu\alpha)^m\delta 
  = \nu (\mu\alpha\nu)^{m-1}\mu\alpha\delta = \nu (X_{s(\mu)})^{m-1}\mu f(\mu)g\big(f(\mu)\big) = 0.
\]
\end{proof}

\begin{lemma}
\label{lem:4.3} 
We have the following equalities in $\Lambda$.
\begin{enumerate}[(i)]
 \item
  $X_1 = \nu\omega\gamma = \delta\xi\alpha$, 
  $X_3 = \sigma\varrho\mu = \alpha\delta\xi$, 
  $X_6 = \mu\sigma\varrho = \omega\gamma\nu$. 
 \item
  $X_2 = \varrho\varrho\sigma$, 
  $X_4 = \gamma\nu\omega$, 
  $X_5 = \xi\alpha\delta$.
 \item
  $X_2 = \tilde{X}_2 + \lambda X_2^m$, 
  $X_4 = \tilde{X}_4 + \lambda X_4^m$, 
  $X_5 = \tilde{X}_5 + \lambda X_5^m$.
 \item
  $X_2^m = (\tilde{X}_2)^m$, 
  $X_4^m = (\tilde{X}_4)^m$, 
  $X_5^m = (\tilde{X}_5)^m$.
\end{enumerate}
\end{lemma}

\begin{proof} 
The equalities in (i) and (ii) follow directly from the relations
defining $\Lambda$.  
For (iii), observe that the vertices $2$, $4$, $5$
are in one orbit of the autmorphism $\varphi$. 
Hence, it is enough to show that
$X_2 = \tilde{X}_2 + \lambda X_2^m$.
We have
\[
  X_2 = \varepsilon\xi\sigma = \rho\mu\sigma = \rho\omega\beta .
\]
Moreover
\[
 \rho\omega\beta
 = (\varepsilon\eta + \lambda X_2^{m-1}\varepsilon\eta)\beta 
 = \tilde{X}_2 + \lambda X_2^{m-1}\varepsilon\eta\beta
\]
and 
\[
 X_2^{m-1}\varepsilon\eta\beta
 = X_2^{m-1}\varepsilon(\xi\sigma - \lambda X_5^{m-1}\eta\beta)
 = X_2^{m-1}\varepsilon\xi\sigma
 = X_2^m.
\]
The equalities in (iv) follow from the equalities in (iii)
and the fact that 
$X_2^m$, $X_4^m$, $X_5^m$ are in the socle of $\Lambda$.  
\end{proof}

\begin{lemma}
\label{lem:4.4} 
For vertices $i \neq j$ in $Q$, any two paths of length $3$
from $i$ to $j$ are equal and non-zero in $\Lambda$.  
\end{lemma}

\begin{proof} 
Consider paths of length three between different vertices $i$ and $j$ in $Q$. 
Such paths only exist if the vertices are ``opposite'', and
because of the automorphism $\varphi$, we may assume 
that $\{ i, j\} = \{ 1, 2\}$. 
Concerning paths from $1$ to $2$ we have
\[
 \delta\eta\beta = \delta\xi\sigma - \lambda\delta X_5^{m-1}\eta\beta.
\]
Now, $\delta X_5 = \delta\eta\gamma\delta = X_1\delta$ and therefore
\[
 \delta X_5^{m-1}\eta\beta
 = X_1^{m-1}\delta\eta\beta
 = X_{s(\delta)}^{m-1}\delta f(\delta)g\big(f(\delta)\big)
 = 0.
\]
With this, we have
\[
 \delta\eta\beta = \delta\xi\sigma = \nu\mu\sigma = \nu\omega\beta ,
\]
as required. 
A similar calculation shows that all paths from 2 to 1 of length three 
are equal in $\Lambda$. 
\end{proof}

\begin{lemma}
\label{lem:4.5} 
The following statements hold:
\begin{enumerate}[(i)]
 \item
  For $4 \leq k \leq 3m - 1$, any two paths of length $k$
  between two vertices in $Q$ are equal and non-zero in $\Lambda$.  
 \item
  For $k = 3m$, any path of length $k$
  between two different vertices is zero in $\Lambda$.  
 \item
  For $k = 3m$, any cycle of length $k$ around a vertex $i$
  is equal to $X_i^m$.  
 \item
  For $k > 3m$, any path of length $k$
  is zero in $\Lambda$.  
\end{enumerate}
\end{lemma}

\begin{proof} 
For the following, we  write $X_{ij}$ for a path of length three between 
vertices $i\neq j$. 
We first show that any two paths of length four
between two fixed vertices are equal. 
For this, it suffices to consider paths starting at $1$ 
and paths starting at $2$. 

\smallskip
 
 (i1) 
Paths from $1$ of length four must end at vertex $5$ 
or vertex $6$. 
Consider paths ending at $5$. 
Such a path either
 ends with arrow $\delta$ or it ends with arrow $\varepsilon$. 
If it ends with $\delta$ then it is the product of a cyclic path of length 
three from $1$ to $1$ with $\delta$, hence by 
Lemma~\ref{lem:4.3}, 
is equal to $X_1\delta$.
Similarly, any path of length four 
from $1$ ending with $\varepsilon$ is the product of a path 
of length three from $1$ to $2$ with $\varepsilon$, 
hence is equal in $\Lambda$ to $X_{12}\varepsilon$. 
We must show that $X_1\delta = X_{12}\varepsilon$. 
We have
\[
 X_1\delta 
 = \nu\mu\alpha\delta
 = \nu\mu\sigma\varepsilon
 = X_{12}\varepsilon.
\]
Similarly, any path of length four from 1 to 6 ends with arrow $\rho$ 
or with arrow $\nu$, and one shows as above that
all are equal in $\Lambda$. 

\smallskip
 
(i2) 
Consider paths of length four starting at vertex $2$, 
any such path ends at vertex $6$ or vertex $5$. 
Consider paths ending  at vertex $6$, 
the last arrow in such a path is $\nu$ or $\rho$. 
If it ends with $\nu$ then the path is of the form 
$X_{12}\nu$, and if it ends with $\rho$ then it is either $X_2\rho$, 
or it is $\tilde{X}_2\rho$. 
We have
\[
  \tilde{X}_2\rho - (X_2 - \lambda X_2^m)\rho = X_2\rho
\]
(noting that $X_2^m$ is in the right socle of $\Lambda$). Moreover, 
\[
  X_2\rho = \rho\omega\beta\rho = \rho\omega\gamma\nu = X_{12}\nu.
\]
For paths ending at vertex $5$ the proof is similar.

\smallskip

We finish the proof of (i)  by induction on $k$, 
using arguments as for the case $k=4$. 
Note that all paths of length $\leq 3m-1$ 
in $Q$ are non-zero in $\Lambda$ since all zero relations 
of $\Lambda$ have length $3m$ 
(and since the relations as listed are minimal).

\smallskip

We prove now the statements (ii) and (iii).
It suffices again to consider paths starting at 1 and paths starting at $2$.
A cyclic path starting at $1$ of length $3m$ is of the form 
$Y\gamma$ or $Y'\alpha$, where $Y$ ends at vertex $4$
and $Y'$ ends at vertex $3$. 
By part (i) we can take $Y=X_1^{m-1}\delta\eta$ 
and then $Y\gamma = X_1^m$. 
As well we can take $Y'=X_1^{m-1}\nu\mu$ and 
get $Y'\alpha = X_1^m$. 
Similarly, any path of length $3m$ from $2$ to $2$ 
is equal to $X_2^m$.
Now consider a path from vertex $1$ of length $3m$ 
which does not end 
at vertex $1$, then it must end at vertex $2$. 
It is  of the form $Y\beta$ with $Y$ from $1$ to $4$, 
or of the form $Y'\sigma$ with $Y'$ from $1$ to $3$. 
By part (i) we can take
$Y=X_1^{m-1}\delta\eta$ and then 
\[
  Y\beta = X_{s(\delta)}^{m-1}\delta f(\delta)g\big(f(\delta)\big) = 0.
\]
We also can take $Y' = X_1^{m-1}\nu\mu$ and then again, by the defining relations, $Y'\sigma = 0$. 
Finally, consider a path from vertex $2$ of length $3m$ 
which does not end at vertex $2$, then it must end at vertex $1$. 
Such a path is either of
the form $Y\gamma$, or of the form $Y'\alpha$, where $Y$ and $Y'$ 
are paths of length $3m-1$. 
We can take $Y=X_2^{m-1}\rho\omega$ and then $Y\gamma=0$, 
by the defining relations. 
Similarly, we can take $Y'=X_2^{m-1}\varepsilon\xi$ 
and then $Y'\alpha=0$, by the defining relations.

\smallskip

The statement (iv) follows because $X_i^m$ is in the right socle 
of $\Lambda$, for any vertex $i$ of $Q$.
\end{proof}

We present now a basis of $\Lambda$ with good properties.
We fix a vertex $i$, and define a basis $\cB_i$ of $e_i\Lambda$ as follows.
Choose a version of $X_i$, then suppose $X_i$ starts with $\tau$, 
then let $\bar{\tau}$ be the other arrow starting at $i$.
Now let 
$\cB_i:= $ the set of   all initial subwords of  $ X_i^m$ together with the set
\[
 \Big\{ X_i^k\bar{\tau}, X_i^k\bar{\tau}f(\bar{\tau}) : 0\leq k\leq m-1\Big\}
 \cup \Big\{ X_i^k\tau f(\tau)g\big(f(\tau)\big): 0\leq k < m-1 \Big\}.
\]
Then $\cB_i$ is a basis for $e_i\Lambda$, 
and we take $\cB:= \cup_{i\in Q_0} \cB_i$. 
For each vertex $i$, let $\omega_i:= X_i^m$, this spans the socle of $e_i\Lambda$,
by Lemma~\ref{lem:4.5},
and it lies in $\cB$. 
The basis $\cB$ has the following properties:
 \begin{enumerate}[(a)]
 \item
   For each $k$ with $1\leq k\leq 3m-1$ the set $\cB_i$ 
   contains precisely two elements of length $k$. 
   The end vertices
   are determined by the congruence of $k$ modulo $3$. 
 \item
   Any path of length $k$ for $4\leq k \leq 3m-1$ is equal 
   to precisely one basis element, 
   as well any path of length three, except  the cyclic paths 
   between vertices $2, 4, 5$.
 \item
   The product of two elements $b, b'$ from $\cB$ is either zero, 
   or is again an element in $\cB$. 
   It is non-zero if and only if $t(b)= s(b')$ 
   and  $bb'$ has  length $\leq 3m$,
   and if the length is $3m$ then $s(b) = t(b')$. 
(For this, note that the cyclic paths of length three through the vertices 
$2, 5, 4$ are not products of basis elements.)
 \item
   For each $b\in \cB_i$ there is a unique $\hat{b}\in \cB$ 
   such that $b\hat{b} = \omega_i$: 
   Say $b=be_j$of length $k$, then $\cB_j$ must contain 
   a unique element $\hat{b}$ of length $3m-k$ and moreover
   which ends at $i$. 
   This is seem by checking through each congruence. 
   Then $b\hat{b}$ is a path of length $3m$
   from $i$ to $i$ and it must therefore be equal 
   to $\omega_i$, by Lemma~\ref{lem:4.5}.
   It must be unique with $bb'=\omega_i$ and $b'\in \cB$.
\end{enumerate}

\begin{corollary}
\label{cor:4.6} 
$\Lambda$ has dimension $36m$.
\end{corollary}

The next theorem completes the proof of Theorem~\ref{th:main1}. 

\begin{theorem}
\label{th:4.7} 
$\Lambda$ is a symmetric algebra.
\end{theorem}

\begin{proof} 
We use the above basis  to define a symmetrizing bilinear form.
If $b, b'\in \cB$, define 
\[
  (b, b'):= \mbox{ the coefficient of $\omega_i$ when $bb'$ 
                           is expanded in terms of $\cB$. }
\]
This extends to a bilinear form, and it is clearly associative. 
By (c) and (d) above, the Gram matrix of the bilinear form 
is non-singular, hence the form is non-degenerate.
We show that the form is symmetric.

Let $b, b'\in \cB$, where $b= e_ibe_j$ and $b'= e_kb'e_l$. Then we have
\[
  (b, b')  = \left\{\begin{array}{ll} 
     1 & j=k, i=l, \ell(bb') = 3m,\cr
     0 & \mbox{else},
   \end{array}\right.
\]
and $(b, b')$ is the same. 
\end{proof}

\section{Proof of Theorem~\ref{th:main2}}\label{sec:proof2}

Let $(Q,f)$ be the triangulation quiver
associated to the tetrahedron.
Then we have the involution $\bar{}: Q_1 \to Q_1$
on the set $Q_1$ of arrows of $Q$ which assigns to an arrow
$\theta \in Q_1$ the arrow $\bar{\theta}$ with
$s(\theta) = s(\bar{\theta})$ and $\theta \neq \bar{\theta}$.
With this, we obtain another permutation $g: Q_1 \to Q_1$
such that
$g(\theta) = \overbar{f(\theta)}$ for any $\theta \in Q_1$,
as indicated in the introduction.

Let $m \geq 2$ be a natural number, $\lambda \in K$, and
$\Lambda(m,\lambda)$ the associated 
higher tetrahedral algebra.
We will prove first that $\Lambda(m,\lambda)$ is a tame algebra.
We divide the proof into several steps.

\begin{proposition}
\label{prop:5.1}
For each $\lambda \in K \setminus \{ 0 \}$, $\Lambda(m,\lambda)$ 
degenerates to $\Lambda(m,0)$.
\end{proposition}

\begin{proof}
For each $t \in K$, consider the algebra $\Lambda(t)$
given by the quiver $Q$ and the relations:
\begin{align*}
 \gamma\delta &= \beta\varepsilon + t \lambda (\beta\varrho\omega)^{m-1} \beta\varepsilon, 
 &
 \delta\eta &= \nu\omega,
 &
 \eta\gamma &= \xi\alpha,
 & 
 \nu \mu &= \delta\xi ,
\\
 \varrho\omega &= \varepsilon\eta + t \lambda (\varepsilon\xi\sigma)^{m-1} \varepsilon\eta,
 &
 \omega\beta &= \mu\sigma,
 &
 \beta\varrho &= \gamma\nu ,
 &
 \mu \alpha &= \omega \gamma ,
\\
 \xi\sigma &= \eta\beta + t \lambda (\eta\gamma\delta)^{m-1} \eta\beta,
 &
 \sigma\varepsilon &= \alpha\delta,
 &
 \varepsilon\xi &= \varrho\mu,
 &
 \alpha\nu &= \sigma\varrho, 
\\
\omit\rlap{\qquad\quad\ \ $\big(\theta f(\theta) f^2(\theta)\big)^{m-1} \theta f(\theta) g\big(f(\theta)\big) = 0$ for any arrow $\theta$ in $Q$.}
\end{align*}
Then $\Lambda(t)$, $t \in K$, is an algebraic family
in the variety $\alg_d(K)$, with $d = 36m$.
Observe that $\Lambda(0) \cong \Lambda(m,0)$ and 
$\Lambda(1) \cong \Lambda(m,\lambda)$.
Fix $t \in K \setminus \{ 0 \}$, and take an element $a_t \in K$
with $a_t^{3(m-1)} = t$.
Then there is an isomorphism of algebras  
$\varphi_t : \Lambda(1) \to \Lambda(t)$
such that
$\varphi_t(\theta) = a_t \theta$ for any arrow $\theta$ in $Q$.
This shows that $\Lambda(t) \cong \Lambda(1)$ for all $t \in K \setminus \{ 0 \}$.
Then it follows from Proposition~\ref{prop:2.2}
that $\Lambda(m,\lambda)$  degenerates to $\Lambda(m,0) = \Lambda(0)$.
\end{proof}

Let $\Omega(m)$ be the algebra given by quiver $\Delta$ of the form
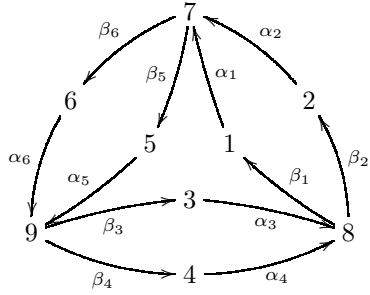
\[
\begin{xy}
0;/r2.5pc/:
(0.5,1.125)*+{1}="1";
(1.5,1.675)*+{2}="2";
(0,0.425)*+{3}="3";
(0,-0.525)*+{4}="4";
(-0.5,1.125)*+{5}="5";
(-1.5,1.675)*+{6}="6";
(0,2.8)*+{7}="7";
(2,0)*+{8}="8";
(-2,0)*+{9}="9";
\ar @{->}@/^.25ex/_{\alpha_1} "1";"7"
\ar @{->}@/_1ex/_{\alpha_2} "2";"7"
\ar @{->}@/^.5ex/_{\alpha_3} "3";"8"
\ar @{->}@/_1ex/_{\alpha_4} "4";"8"
\ar @{->}@/^.5ex/_{\alpha_5} "5";"9"
\ar @{->}@/_1ex/_{\alpha_6} "6";"9"
\ar @{->}@/^.5ex/_{\beta_5} "7";"5"
\ar @{->}@/_1ex/_{\beta_6} "7";"6"
\ar @{->}@/^.5ex/_{\beta_1} "8";"1"
\ar @{->}@/_1ex/_{\beta_2} "8";"2"
\ar @{->}@/^.5ex/_{\beta_3} "9";"3"
\ar @{->}@/_1ex/_{\beta_4} "9";"4"
\end{xy}
\]
and the relations:
\begin{align*}
 \beta_1\alpha_1 &= \beta_2\alpha_2, &
 \beta_3\alpha_3 &= \beta_4\alpha_4, &
 \beta_5\alpha_5 &= \beta_6\alpha_6,
\end{align*}
\begin{align*}
 \alpha_1 (\beta_5 \alpha_5 \beta_3 \alpha_3 \beta_1 \alpha_1)^{m-1}
    \beta_5 \alpha_5 \beta_3 \alpha_3 \beta_2 &= 0, 
& 
 \alpha_2 (\beta_6 \alpha_6 \beta_4 \alpha_4 \beta_2 \alpha_2)^{m-1}
    \beta_6 \alpha_6 \beta_4 \alpha_4 \beta_1 &= 0, 
\\
 \alpha_3 (\beta_1 \alpha_1 \beta_5 \alpha_5 \beta_3 \alpha_3)^{m-1}
    \beta_1 \alpha_1 \beta_5 \alpha_5 \beta_4 &= 0, 
& 
 \alpha_4 (\beta_2 \alpha_2 \beta_6 \alpha_6 \beta_4 \alpha_4)^{m-1}
    \beta_2 \alpha_2 \beta_6 \alpha_6 \beta_3 &= 0, 
\\ 
 \alpha_5 (\beta_3 \alpha_3 \beta_1 \alpha_1 \beta_5 \alpha_5)^{m-1}
    \beta_3 \alpha_3 \beta_1 \alpha_1 \beta_6 &= 0, 
& 
 \alpha_6 (\beta_4 \alpha_4 \beta_2 \alpha_2 \beta_6 \alpha_6)^{m-1}
    \beta_4 \alpha_4 \beta_2 \alpha_2 \beta_5 &= 0. 
\end{align*}
For each vertex $i$ of $\Delta$, we denote by $e_i$ the primitive
idempotent of $\Omega(m)$  associated to $i$.
Moreover, let
$e = e_1 + e_2 + e_3 + e_4 + e_5 + e_6$.

\begin{lemma}%
\label{lem:5.2}% 
The following statements hold:
\begin{enumerate}[(i)]
 \item
  $\Omega(m)$ is a finite-dimensional algebra 
  with $\dim_K \Omega(m) = 81 m + 3$.
 \item
  $\Lambda(m,0)$ is isomorphic to the idempotent algebra
  $e \Omega (m) e$.
\end{enumerate}
\end{lemma}

\begin{proof} 
(i)
A direct checking shows that 
$\dim_K e_i \Omega(m) = 9 m$ for $i \in \{1,2,3,4,5,6\}$,
and $\dim_K e_j \Omega(m) = 9m+1$ for $j \in \{7,8,9\}$.
Therefore, we obtain $\dim_K \Omega(m) = 81 m + 3$.

\smallskip

(ii)
Consider the paths of length $2$ in $\Delta$
\begin{align*}
 \delta &= \alpha_1 \beta_5, &
 \nu  &= \alpha_1 \beta_6, &
 \varepsilon  &= \alpha_2 \beta_5, &
 \varrho  &= \alpha_2 \beta_6, & % \\
 \alpha  &= \alpha_3 \beta_1, &
 \sigma  &= \alpha_3 \beta_2, \\
 \gamma  &= \alpha_4 \beta_1, &
 \beta  &= \alpha_4 \beta_2, & % \\
 \xi  &= \alpha_5 \beta_3, &  
 \eta  &= \alpha_5 \beta_4, &
 \mu  &= \alpha_6 \beta_3, &
 \omega  &= \alpha_6 \beta_4.
\end{align*}
Then these paths satisfy the relations defining the algebra
$\Lambda(m,0)$.
Therefore, $e \Omega (m) e$ is isomorphic to $\Lambda(m,0)$.
\end{proof}

The algebra $\Omega (m)$ can be viewed as a blowup
of the algebra $\Lambda(m,0)$.
The reason to consider it here is as follows.
The higher tetrahedral  algebras $\Lambda(m,\lambda)$
have no visible degenerations to special biserial alebras.
But the algebra $\Omega (m)$ admits a degeneration
to a special biserial algebra, as we will show below.
Then Proposition~\ref{prop:2.1} will imply that 
$\Omega (m)$ is a tame algebra,
and consequently $\Lambda(m,0)$ is a tame algebra
(see \cite[Theorem]{DS0}).

For each $t \in K$, let $\Sigma(m,t)$ be the algebra
given by the quiver $\Sigma$ of the form
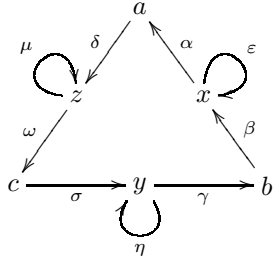
\[
\begin{xy}
0;/r2pc/:
(1,1.4)*+{x}="x";
(0,0)*+{y}="y";
(-1,1.4)*+{z}="z";
(0,2.8)*+{a}="a";
(2,0)*+{b}="b";
(-2,0)*+{c}="c";
\ar @{->}_{\alpha} "x";"a"
\ar @{->}_{\beta} "b";"x"
\ar @{->}_{\gamma} "y";"b"
\ar @{->}_{\sigma} "c";"y"
\ar @{->}_{\omega} "z";"c"
\ar @{->}_{\delta} "a";"z"
\ar@(dr,dl)^{\eta} "y";"y"
\ar@(u,r)^{\varepsilon} "x";"x"
\ar@(l,u)^{\mu} "z";"z"
\end{xy}
\]
and the relations:
\begin{align}
 \beta \alpha &=0, &
 \sigma \gamma &= 0, &
 \delta \omega &= 0, &
 \varepsilon^2 &= t \varepsilon, &
 \eta^2 &= t \eta, &
 \mu^2 &= t \mu,
\end{align}
\begin{align}
 t(\alpha\delta\mu\omega\sigma\eta\gamma\beta\varepsilon)^m
 &= \varepsilon(\alpha\delta\mu\omega\sigma\eta\gamma\beta\varepsilon)^m ,
 &
 t(\varepsilon\alpha\delta\mu\omega\sigma\eta\gamma\beta)^m
 &= (\varepsilon\alpha\delta\mu\omega\sigma\eta\gamma\beta)^m \varepsilon,
\\
\notag
 t(\gamma\beta\varepsilon\alpha\delta\mu\omega\sigma\eta)^m
 &= \eta(\gamma\beta\varepsilon\alpha\delta\mu\omega\sigma\eta)^m ,
 &
 t(\eta\gamma\beta\varepsilon\alpha\delta\mu\omega\sigma)^m
 &= (\eta\gamma\beta\varepsilon\alpha\delta\mu\omega\sigma)^m \eta,
\\
\notag
 t(\omega\sigma\eta\gamma\beta\varepsilon\alpha\delta\mu)^m
 &= \mu(\omega\sigma\eta\gamma\beta\varepsilon\alpha\delta\mu)^m ,
 &
 t(\mu\omega\sigma\eta\gamma\beta\varepsilon\alpha\delta)^m
 &= (\mu\omega\sigma\eta\gamma\beta\varepsilon\alpha\delta)^m \mu,
\end{align}
\begin{align}
 (\alpha\delta\mu\omega\sigma\eta\gamma\beta\varepsilon)^m
 &= (\varepsilon\alpha\delta\mu\omega\sigma\eta\gamma\beta)^m ,
%\\
%\notag
&
 (\gamma\beta\varepsilon\alpha\delta\mu\omega\sigma\eta)^m
 &= (\eta\gamma\beta\varepsilon\alpha\delta\mu\omega\sigma)^m,
\\
\notag
 (\omega\sigma\eta\gamma\beta\varepsilon\alpha\delta\mu)^m
 &= (\mu\omega\sigma\eta\gamma\beta\varepsilon\alpha\delta)^m ,
\end{align}
\begin{align}
 (\delta\mu\omega\sigma\eta\gamma\beta\varepsilon\alpha)^m \delta &=0 ,
 &
 \alpha (\delta\mu\omega\sigma\eta\gamma\beta\varepsilon\alpha)^m &=0 ,
 &
 (\beta\varepsilon\alpha\delta\mu\omega\sigma\eta\gamma)^m \beta &=0 ,
 \\
\notag
 \gamma (\beta\varepsilon\alpha\delta\mu\omega\sigma\eta\gamma)^m &=0 ,
 &
 (\sigma\eta\gamma\beta\varepsilon\alpha\delta\mu\omega)^m \sigma &=0 ,
 &
 \omega (\sigma\eta\gamma\beta\varepsilon\alpha\delta\mu\omega)^m &=0 .
\end{align}
We note that for $t \in K \setminus \{ 0 \}$ 
the relations (3) follow from the relations (2),
and the relationts (4) from the relations (1) and (2).
For example, we have the equalities
\begin{align*}
 t (\delta\mu\omega\sigma\eta\gamma\beta\varepsilon\alpha)^m \delta 
 &= t \delta (\mu\omega\sigma\eta\gamma\beta\varepsilon\alpha\delta)^m 
  = \delta (\mu\omega\sigma\eta\gamma\beta\varepsilon\alpha\delta)^m \mu
\\
 &= \delta \mu (\omega\sigma\eta\gamma\beta\varepsilon\alpha\delta\mu)^m 
  = t \delta (\omega\sigma\eta\gamma\beta\varepsilon\alpha\delta\mu)^m 
 = 0 ,
\end{align*}
because $\delta \omega = 0$, and hence 
$(\delta\mu\omega\sigma\eta\gamma\beta\varepsilon\alpha)^m \delta = 0$,
for $t \in K \setminus \{ 0 \}$.
For each vertex $i$ of $\Sigma$, we denote by $f_i$ the primitive
idempotent of $\Sigma(m,t)$ associated to $i$. 

\begin{lemma}
\label{lem:5.3} 
The following statements hold:
\begin{enumerate}[(i)]
 \item
  For each $t \in K$,
  $\Sigma(m,t)$ is a finite-dimensional algebra 
  with $\dim_K \Sigma(m,t) = 81 m + 3$.
 \item
  $\Sigma(m,t) \cong \Sigma(m,1)$
  for any $t \in K \setminus \{ 0 \}$.
 \item
  $\Sigma(m,0)$ is a special biserial algebra.
\end{enumerate}
\end{lemma}

\begin{proof} 
(i)
It follows from the relations defining
$\Sigma(m,t)$ that $\dim_K f_i \Sigma(m,t) = 9 m + 1$ for $i \in \{a,b,c\}$,
and $\dim_K f_j \Sigma(m,t) = 18 m$ for $j \in \{x,y,z\}$.
Hence, we obtain $\dim_K \Sigma(m,t) = 81 m + 3$.

\smallskip

(ii)
Fix $t \in K \setminus \{ 0 \}$, and take an element $b_t \in K$
with $b_t^8 = t$.
Then there exists an isomorphism of algebras
$\psi_t : \Sigma(m,1) \to \Sigma(m,t)$
such that 
$\psi_t(\varepsilon) = t^{-1} \varepsilon$,
$\psi_t(\eta) = t^{-1} \eta$,
$\psi_t(\mu) = t^{-1} \mu$,
and
$\psi_t(\theta) =b_ t \theta$
for any arrow $\theta \in \{\alpha, \beta, \gamma, \sigma, \omega, \delta \}$.

\smallskip

(iii) 
Follows from the relations defining $\Sigma(m,0)$.
\end{proof}

\begin{lemma}
\label{lem:5.4}
The algebras $\Omega(m)$ and $\Sigma(m,1)$ 
are isomorphic.
\end{lemma}

\begin{proof} 
We shall prove that there is a well defined isomorphism of algebras
$\varphi : \Omega(m) \to \Sigma(m,1)$
such that 
\begin{align*}
 \varphi(e_1) &= \varepsilon, &
 \varphi(e_2) &= f_x - \varepsilon, &
 \varphi(e_3) &= \eta, &
 \varphi(e_4) &= f_y - \eta, \\
 \varphi(e_5) &= \mu, &
 \varphi(e_6) &= f_z - \mu &
 \varphi(e_7) &= f_a, &
 \varphi(e_8) &= f_b, &
 \!\!\!\!\!\!\!\!\!\!\!\!\!\!\!\!\!
 \varphi(e_9) &= f_c, \\
 \varphi(\alpha_1) &= \varepsilon \alpha, &
 \varphi(\alpha_2) &= \alpha - \varepsilon \alpha, &
 \varphi(\beta_1) &= \beta \varepsilon, &
 \varphi(\beta_2) &= - \beta + \beta\varepsilon, \\ 
 \varphi(\alpha_3) &= \eta \gamma, &
 \varphi(\alpha_4) &= \gamma - \eta \gamma, &
 \varphi(\beta_3) &= \sigma \eta, &
 \varphi(\beta_4) &= - \sigma + \sigma \eta, \\ 
 \varphi(\alpha_5) &= \mu \omega, &
 \varphi(\alpha_6) &= \omega - \mu \omega, &
 \varphi(\beta_5) &= \delta\mu, &
 \varphi(\beta_6) &= - \delta + \delta\mu. 
\end{align*}
Observe that
\begin{align*}
 \varphi(e_1 + e_2) &= f_x, &
 \varphi(e_3 + e_4) &= f_y, &
 \varphi(e_5 + e_6) &= f_z, \\
 \varphi(\alpha_1 + \alpha_2) &= \alpha , &
 \varphi(\alpha_3 + \alpha_4) &= \gamma , &
 \varphi(\alpha_5 + \alpha_6) &= \omega , \\
 \varphi(\beta_1 - \beta_2) &= \beta , & 
 \varphi(\beta_3 - \beta_4) &= \sigma , & 
 \varphi(\beta_5 - \beta_6) &= \delta . 
\end{align*}
We have in $\Sigma(m,1)$ the following equalities
\begin{align*}
 \varphi(e_1^2) &= \varphi(e_1) = \varepsilon = \varepsilon^2  = \varphi(e_1)^2, \\
 \varphi(e_2^2) &= \varphi(e_2) =  f_x - \varepsilon = (f_x - \varepsilon)^2  = \varphi(e_2)^2, \\
 \varphi(e_3^2) &= \varphi(e_3) = \eta = \eta^2  = \varphi(e_3)^2, \\
 \varphi(e_4^2) &= \varphi(e_4) =  f_y - \eta = (f_y - \eta)^2  = \varphi(e_4)^2, \\
 \varphi(e_5^2) &= \varphi(e_5) = \mu = \mu^2  = \varphi(e_5)^2, \\
 \varphi(e_6^2) &= \varphi(e_6) =  f_z - \mu = (f_z - \mu)^2  = \varphi(e_6)^2, \\
 \varphi(\beta_1) \varphi(\alpha_1)  &= \beta \varepsilon^2 \alpha
   = \beta \varepsilon \alpha = (- \beta + \beta\varepsilon) (\alpha - \varepsilon \alpha)
   = \varphi(\beta_2) \varphi(\alpha_2), \\
 \varphi(\beta_3) \varphi(\alpha_3)  &= \sigma \eta^2 \gamma
   = \sigma \eta \gamma = (- \sigma + \sigma \eta) (\gamma - \eta \gamma)
   = \varphi(\beta_4) \varphi(\alpha_4), \\
 \varphi(\beta_5) \varphi(\alpha_5)  &= \delta\mu^2 \omega
   = \delta\mu \omega = (- \delta + \delta\mu) (\omega - \mu \omega)
   = \varphi(\beta_6) \varphi(\alpha_6) .
\end{align*}
It remains to show that the six zero relations defining $\Omega(m)$
correspond via $\varphi$ to the six commutativity relations (2),
with $t=1$, defining $\Sigma(m,1)$.
We will show this for the first two relations,
because the proof for the other four is similar.

We have the equalities
\begin{align*}
 \varphi(&\alpha_1) \big(\varphi(\beta_5) \varphi(\alpha_5) \varphi(\beta_3)
          \varphi(\alpha_3) \varphi(\beta_1) \varphi(\alpha_1)\big)^{m-1}
       \varphi(\beta_5) \varphi(\alpha_5) \varphi(\beta_3) \varphi(\alpha_3) \varphi(\beta_2) \\
   &= \varepsilon \alpha (\delta\mu^2 \omega \sigma \eta^2 \gamma \beta \varepsilon^2 \alpha)^{m-1}
         \delta\mu^2 \omega \sigma \eta^2 \gamma (- \beta + \beta\varepsilon) \\
   &= - \varepsilon \alpha (\delta\mu \omega \sigma \eta \gamma \beta \varepsilon \alpha)^{m-1}
          \delta\mu \omega \sigma \eta \gamma \beta
      + \varepsilon \alpha (\delta\mu \omega \sigma \eta \gamma \beta \varepsilon \alpha)^{m-1}
         \delta\mu \omega \sigma \eta \gamma \beta \varepsilon \\
   &= - (\varepsilon \alpha \delta\mu \omega \sigma \eta \gamma \beta)^{m}
      + (\varepsilon \alpha \delta\mu \omega \sigma \eta \gamma \beta)^{m} \varepsilon 
      = 0, 
\\
 \varphi(&\alpha_2) \big(\varphi(\beta_6) \varphi(\alpha_6) \varphi(\beta_4)
          \varphi(\alpha_4) \varphi(\beta_2) \varphi(\alpha_2)\big)^{m-1}
       \varphi(\beta_6) \varphi(\alpha_6) \varphi(\beta_4) \varphi(\alpha_4) \varphi(\beta_1) \\
   &= \varphi(\alpha_2) \big(\varphi(\beta_5) \varphi(\alpha_5) \varphi(\beta_3)
          \varphi(\alpha_3) \varphi(\beta_1) \varphi(\alpha_1)\big)^{m-1}
       \varphi(\beta_5) \varphi(\alpha_5) \varphi(\beta_3) \varphi(\alpha_3) \varphi(\beta_1) \\
   &= (\alpha - \varepsilon \alpha) (\delta\mu \omega \sigma \eta \gamma \beta \varepsilon \alpha)^{m-1}
         \delta\mu \omega \sigma \eta \gamma \beta \varepsilon \\
   &=  (\alpha \delta\mu \omega \sigma \eta \gamma \beta \varepsilon)^{m}
      - \varepsilon (\alpha \delta\mu \omega \sigma \eta \gamma \beta \varepsilon)^{m} 
      = 0. 
\end{align*}
\end{proof}

\begin{corollary}
\label{cor:5.5}
The algebra $\Omega(m)$ degenerates to the
special biserial algebra $\Sigma(m,0)$.
In particular,  $\Omega(m)$ is a tame algebra.
\end{corollary}

\begin{proof} 
It follows from Lemmas \ref{lem:5.3} and \ref{lem:5.4} that
$\Sigma(m,1)$, $t \in K$, is an alebraic family in the variety
$\alg_K(d)$ with $d = 81 m + 3$
such that
$\Sigma(m,1)\cong \Sigma(m,1) \cong \Omega(m)$
for any $t \in K \setminus \{ 0 \}$ and
$\Sigma(m,0)$ is a special biserial algebra.
Then it follows from Propositions \ref{prop:2.1} and \ref{prop:2.2} 
that $\Omega(m)$ is a tame algebra.
\end{proof}

\begin{proposition}
\label{prop:5.6}
For each $\lambda \in K$, $\Lambda(m,\lambda)$ 
is a tame algebra of non-polynomial growth.
\end{proposition}

\begin{proof}
It follows from
Lemma \ref{lem:5.2}~(ii), Corollary~\ref{cor:5.5}
and \cite[Theorem]{DS0} that $\Lambda(m,0)$ is a tame algebra.
Then, applying 
Propositions \ref{prop:2.2} and \ref{prop:5.1},
we conclue that 
$\Lambda(m,\lambda)$ is a tame algebra
for any $\lambda \in K \setminus \{ 0 \}$.
 $\Lambda = \Lambda(m,\lambda)$ 
for an arbitrary $\lambda \in K$.
Consider now the quotient algebra $\Gamma$ of $\Lambda$
by the ideal generated by the arrows $\delta$, $\nu$, $\varepsilon$, $\varrho$.
Then $\Gamma$ is the algebra given by the quiver
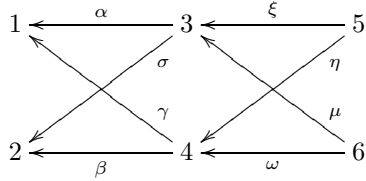
\[
%  \xymatrix@C=1.5pc{
  \xymatrix@C=4.5pc@R=3pc{
    1 &
    3 \ar[l]_{\alpha} \ar[ld]^(.2){\sigma} &
    5 \ar[l]_{\xi} \ar[ld]^(.2){\eta}
    \\
    2 &
    4 \ar[l]^{\beta} \ar[lu]_(.2){\gamma} &
    6 \ar[l]^{\omega} \ar[lu]_(.2){\mu}
  }
\]
and the relations
\begin{align*}
 &&
 \omega \beta &= \mu \sigma,
 &
 \eta \gamma &= \xi \alpha,
 &
 \mu \alpha &= \omega \gamma,
 &
 \xi \sigma &= \eta \beta.
 &&
\end{align*}
Then $\Gamma$ is the tame minimal non-polynomial growth algebra
$(30)$ from \cite{NoS}.
Therefore, $\Lambda$ is of non-polynomial growth.
\end{proof}

We end this section with a Galois covering interpretation of the singular
higher tetrahedral algebras.

Let $m \geq 2$ be a natural number.
We denote by $B(m)$ the fully commutative algebra
of the following quiver
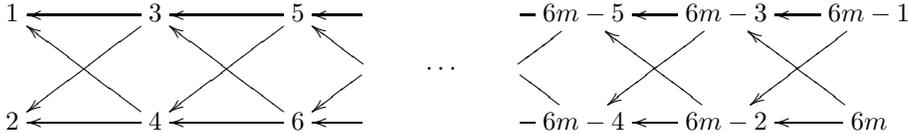
\[
%  \xymatrix@C=1.5pc{
  \xymatrix@C=1.75pc@R=1.2pc@!=.5pc{
    1 &&
    3 \ar[ll] \ar[lldd] &&
    5 \ar[ll] \ar[lldd] & \ar@{->}[l] &
    &&
    6m-5 \ar@{-}[l] \ar@{-}[ld] &&
    6m-3 \ar[ll] \ar[lldd] &&
    6m-1 \ar[ll] \ar[lldd]
    \\
    && && &  \ar@{->}[lu] \ar@{->}[ld] & \cdots &
    \\
    2 &&
    4 \ar[ll] \ar[lluu] &&
    6 \ar[ll] \ar[lluu] & \ar@{->}[l] &
    &&
    6m-4 \ar@{-}[l] \ar@{-}[lu] &&
    6m-2 \ar[ll] \ar[lluu] &&
    6m \ar[ll] \ar[lluu]
  }
\]
Consider the repetitive category $\widehat{B(m)}$ of $B(m)$.
Then the Nakayama automorphism 
$\nu_{\widehat{B(m)}}$ of $\widehat{B(m)}$
admits an $m$-th root $\varphi_m$ such that
$(\varphi_m)^m = \nu_{\widehat{B(m)}}$.
Let $\Gamma(m)$ be the orbit algebra 
$\widehat{B(m)}/(\varphi_m)$ of $\widehat{B(m)}$ 
with respect to the infinite cyclic group $(\varphi_m)$
generated by $\varphi_m$
(see \cite{Sk2} for relevant definitions).

Then we obtain the following proposition.

\begin{proposition}
\label{prop:5.7}
The algebras $\Lambda(m,0)$ and $\Gamma(m)$ are isomorphic.
\end{proposition}

We would like to strees that, for any $\lambda \in K \setminus \{ 0 \}$,
the non-singular higher tetrahedral algebra 
$\Lambda(m,\lambda)$ is not the orbit algebra of
the repetitive category of an algebra.

\section{Proof of Theorem~\ref{th:main3}}\label{sec:proof3}

We show first that every simple $\Lambda$-module 
is periodic of period four. 
This will then tell us what the terms of a minimal 
projective bimodule resolution of $\Lambda$ must be
(see Proposition~\ref{prop:3.1}).
As for notation, we write $\Omega$ for  syzygies 
of right $\Lambda$-modules,   
and we write $\Omega_{\Lambda^e}$ for syzygies
of right $\Lambda^e$-modules ($\Lambda$-$\Lambda$-bimodules). 

\begin{proposition}
\label{prop:6.1} 
Each simple $\Lambda$-module is periodic of period four. 
There is an exact sequence
\[
  0\to S_i \to P_i \to  P_x\oplus P_y \ \to P_j\oplus P_k \to P_i \to S_i \to 0
\]
where the arrows adjacent to $i$ end at $j, k$ and start at $x, y$.
\end{proposition}

\begin{proof}
The automorphism  $\varphi$ of $\Lambda$  induces 
an equivalence of the module category $\mod \Lambda$, 
with two orbits on simple modules. 
We only need to prove periodicity for one simple from each orbit. 
We will consider $S_1$ and $S_4$.

%\bigskip
%\smallskip
\medskip

(1) 
We compute $\Omega^2(S_1)$ which we identify with the kernel 
of the map $d_1: P_6\oplus P_5\to P_1$ defined by
\[
  d_1(a, b) := \nu a + \delta b,
\]
for $a \in P_6$ and $b \in P_5$.
Since $\nu\omega = \delta\eta$ and $\nu \mu = \delta \xi$, 
the kernel contains the submodule 
generated by $\phi$ and $\psi$, where
\[
  \phi = (-\omega, \eta) \ \ \mbox{ and } \ \ \psi = (\mu, -\xi).
\]
We will show that $\Ker d_1 = \phi\Lambda + \psi \Lambda$.  
Since we have one inclusion, 
it suffices to show that both spaces have the same dimension, 
that is,  we must show
that $\phi \Lambda + \psi \Lambda$ has dimension $6m+1$. 
We observe that $\phi \Lambda$ is isomorphic to 
$\Omega^{-1}(S_4)$ since $\omega, \eta$ are
the arrows ending at vertex $4$. 
Similarly, $\psi \Lambda \cong \Omega^{-1}(S_3)$. 
In particular,
$\dim_K \phi \Lambda  = 6m-1 = \dim_K \psi \Lambda$. 
It follows that we must show that  
$\dim_K \phi \Lambda \cap \psi \Lambda = 6m-3$, that is, 
\[
  \dim_K \phi \Lambda/(\phi \Lambda\cap \psi \Lambda) = 2.
\]

(1a) 
We identify  the intersections of $\phi \Lambda$ and $\psi \Lambda$ 
with $0\oplus P_5$. 
We claim that each of $\phi \Lambda \cap (0\oplus P_5)$ 
and $\psi \Lambda\cap (0\oplus P_5)$ is 1-dimensional, 
spanned by $( 0, X_5^m)$.
Indeed, 
suppose $\phi p = (0, z)$ for some $p\in \Lambda$ and $0\neq z$. 
We may assume that $p$ is a monomial in the arrows. 
To have $\omega p=0$ the monomial $p$ must have
length $\geq 3m-1$. 
To have $\omega p =0$ and $\eta p = z\neq 0$, 
we must have that $p$ has length
$3m-1$ and ends at vertex $4$, 
and then $\eta p = X_5^m$. 
For the converse, take $p= X_4^{m-1}\gamma\delta$. 
Similarly one proves the second statement.

%\bigskip
%\smallskip
\medskip

(1b) We claim that $\phi J^2 + \phi \gamma K$ is contained 
in the intersection $\phi \Lambda\cap \psi \Lambda$. 
Namely, we have  $\phi \gamma = -\psi \alpha$, by the relations. 
Next, we have
\begin{align}
 \tag{*}
 \label{eq:*0}
 \phi \beta
  = (-\omega\beta, \eta\beta)
  = -\psi\sigma - (0, \lambda X_5^{m-1}\eta\beta).
\end{align}
Hence 
$\phi\beta\varepsilon = \psi\sigma\varepsilon - (0, \lambda X_5^m)$ 
(using Lemmas \ref{lem:4.3}, \ref{lem:4.4}, \ref{lem:4.5}).  
By (1a) above, this belongs to the intersection
and it follows from these that 
$\phi J^2 \subseteq \phi \Lambda \cap \psi \Lambda$.
We note that if $\lambda = 0$, then $\Omega^2(S_i)$
has more than two minimal generators,
and hence $S_i$ is not periodic of period $4$.
  
%\bigskip
%\smallskip
\medskip

(1c) 
Note that $\phi J^2 +  \phi\gamma K$ has dimension $6m-3$. 
We have the chain 
of submodules
\[
  \phi J^2 + \phi \gamma K 
   \subseteq \phi \Lambda\cap \psi \Lambda
   \subseteq \phi \Lambda ,
\]
and the quotient $\phi \Lambda /(\phi J^2 +  \phi \gamma K)$ 
is spanned by the cosets of $\phi$ and $\phi \beta$. 

Assume for a contradiction 
that $\phi\beta \in \phi \Lambda \cap \psi \Lambda$. 
Then, by (\ref{eq:*0}), we have 
$(0,X_5^{m-1}\eta\beta) \in \psi \Lambda$, but this contradicts (1a). 
So $\phi\beta$ is not in the intersection, and therefore the dimension 
of $\phi \Lambda/ (\phi \Lambda\cap \psi \Lambda)$ is $2$, as required.

%\bigskip
%\smallskip
\medskip

(1d) 
Now it is easy to see that $S_1$ has period four. 
Namely, define $d_2: P_4\oplus P_3\to \Omega^2(S_1)$ by
\[
  d_2(u, v):= \phi u + \psi v,
\]
for $u \in P_4$ and $v \in P_3$.
The kernel of $d_2$, that is, $\Omega^3(S_1)$ has dimension 
$2(6m)-(6m+1) = 6m-1$. 
We have seen that $\phi \gamma = - \psi\alpha$, and
therefore $(\gamma, \alpha)\Lambda\subseteq \Ker(d_2)$. 
This submodule is isomorphic to 
$\Omega^{-1}(S_1)$ and has dimension $6m-1$. 
We deduce that
\[
  \Omega^{-1}(S_1) \cong (\gamma, \alpha)\Lambda \cong \Omega^3(S_1).
\]
So $S_1$ is periodic of period dividing $4$, and then equal to $4$. 

%\bigskip
%\smallskip
\medskip

(2) 
We compute $\Omega^2(S_4)$, which we identify with 
the kernel of $d_1: P_1\oplus P_2\to P_4$ defined as
\[
  d_1(w,z) = \gamma w + \beta z,
\]
for $w \in P_1$ and $z \in P_2$.
This is analogous to (1), there is only a small difference in the formulae.
Using the relations, the kernel of $d_1$ contains $\phi$ and $\psi$,
where
\[
  \phi
  = \big(\delta, -\varepsilon -\lambda(\rho\omega\beta)^{m-1}\varepsilon\big) 
   \ \ \mbox{ and} \  \psi = (-\nu, \rho).
\]
By the same arguments as in (1), to prove that 
$\Ker(d_1) = \phi \Lambda + \psi \Lambda$, 
we must show that
$\dim_K \phi \Lambda/(\phi \Lambda\cap \psi \Lambda) = 2$.
We have $\phi \eta = -\psi \omega$, which is in the intersection, 
and we have
\[
  \phi\xi = -\psi\mu - (0, \lambda X_2^{m-1}\varepsilon\xi) .
\]
As before one shows that
$\phi J^2 = \psi J^2$ and hence is in the intersection. 
Suppose $\phi\xi $ is in the intersection. 
Then it follows that 
$(0, -\lambda X_2^{m-1}\varepsilon\xi)$ is in $\psi \Lambda$, 
which is a contradiction to the analogue of
(1a). 
It follows that $\phi \Lambda/\phi \Lambda\cap \psi \Lambda$ 
is 2-dimensional. 
Then as in (1d) one concludes that $S_4$ has $\Omega$-period four.
\end{proof}

We use the notation as in Section~\ref{sec:bimodule}, 
in particular the description of  
$\bP_0$ and $\bP_1$.
For the higher tetrahedral algebra, we need to specify $\bP_2$, 
which has generators
corresponding to the minimal relations involving paths of length two.
Each of these minimal relations has a term 
$\theta f(\theta)$ for $\theta$ an arrow, and this gives
a bijection between arrows and minimal relations involving 
paths of length two.
So we take 
\[
  \bP_2:= \oplus_{\theta \in Q_1} 
      \Lambda(e_{s(\theta)}\otimes e_{t(f(\theta))})\Lambda.
\]
We may denote the minimal relation with term $\theta f(\theta)$  
by $\mu_{\theta}$. Then the definition of
$R$ in Section~\ref{sec:bimodule} specializes to 
\[
 R: \bP_2\to \bP_1, \ \ R(e_{s(\theta)}\otimes e_{t(f(\theta))})
  := \pi(\mu_{\theta}).
\]

\begin{lemma}
\label{lemma:6.2} 
The homomorphism $R: \bP_2\to \bP_1$ induces a projective cover of
$\Omega^2_{\Lambda^e}(\Lambda)$ in $\mod \Lambda^e$. 
In particular, $\Omega^3_{\Lambda^e}(\Lambda) = \Ker R$.
\end{lemma}

\begin{proof}
This is similar as that of Lemma~7.2 of \cite{ESk-WSA}. 
\end{proof}

By Propositions \ref{prop:3.1} and \ref{prop:6.1},
we can take 
$\bP_3 = \oplus_{i\in Q_0} \Lambda(e_i\otimes e_i)\Lambda$.  
For each vertex $i$ of $Q$, 
we define an element $\psi_i$ as follows.
Let $\tau, \bar{\tau}$ be the arrows starting at $i$, and let
$\theta, \bar{\theta}$ be the arrows ending at $i$. 
Set
\[
  \psi_i:= (e_i\otimes e_{t(\theta)})\theta 
     + (e_i\otimes e_{t(\bar{\theta})})\bar{\theta}
     - \tau(e_{t(\tau)}\otimes e_i)
     - \bar{\tau}(e_{t(\bar{\tau})}\otimes e_i).
\]
Then we define a $\Lambda^e$-module homomorphism 
$S: \bP_3\to \bP_2$ by 
\[
   S(e_i\otimes e_i):= \psi_i,  \ \mbox{for } i\in Q_0.
\]

\begin{lemma}
\label{lemma:6.3}  
The homomorphism $S:\bP_3\to \bP_2$ 
induces a projective cover of 
$\Omega^3_{\Lambda^e}(\Lambda)$ in $\mod \Lambda^e$. 
In particular, we have 
$\Omega^4_{\Lambda^e}(\Lambda) = \Ker(S).$
\end{lemma}

\begin{proof}
We know that the kernel of $R$ is 
$\Omega^3_{\Lambda^e}(\Lambda)$, and we know that
it has minimal generators corresponding to the vertices of $Q$.
As well, from the definition, the element $\psi_i$ does not lie in  
$(\rad \bP_2)^2$. 
Therefore, it is enough to show that $R(\psi_i) = 0$ for all $i$. 
 
The algebra automorphism $\varphi$  of $\Lambda$ defined in 
Section~\ref{sec:proof1},   
extends to an automorphism of $\Lambda^e$. 
One checks that it  commutes with the map $R$ and that it takes 
$\psi_i$ to $\psi_{\varphi(i)}$. 
So it is enough to take $i=1$ and $i=4$. 

%\bigskip
%\smallskip
\medskip

(1) 
We compute $R(\psi_1)$. This is equal to
\begin{align*}
  R\big((e_1\otimes e_4)\gamma &+ (e_1\otimes e_3)\alpha - \nu(e_6\otimes e_1) - \delta(e_5\otimes e_1)\big)\cr
 = \ & \pi(\delta\eta - \nu\omega)\gamma + \pi(\nu\mu - \delta\xi)\alpha
 -\nu \big(\pi(\mu\alpha -\omega\gamma)\big)
 - \delta\big(\pi(\eta\gamma - \xi\alpha)\big)\cr
 = \ & (e_1\otimes \eta\gamma + \delta\otimes \gamma - e_1\otimes \omega\gamma - \nu\otimes \gamma)\cr
 &+ (e_1\otimes \mu\alpha + \nu\otimes \alpha - e_1\otimes \xi\alpha - \delta\otimes \alpha)\cr
 &- (\nu\otimes \alpha + \nu\mu\otimes e_1 - \nu\otimes \gamma - \nu\omega \otimes e_1)\cr
 & - (\delta\otimes \gamma + \delta\eta\otimes e_1 - \delta\otimes \alpha - \delta\xi\otimes e_1).
\end{align*}
The terms of the form $\alpha_1\otimes \alpha_2$ 
for $\alpha_i$ arrows, cancel. 
The terms in $(e_1\otimes e_5)\Lambda$ are 
\[
 e_1\otimes \eta\gamma - e_1\otimes \xi\alpha)
  = e_1\otimes (\eta\gamma - \xi\alpha)
  = 0 .
\]
Similarly, there are two terms in $(e_1\otimes e_6)\Lambda$ 
and two terms in $\Lambda(e_4\otimes e_1)$ and two terms
 in $\Lambda(e_3\otimes e_1)$, and they all cancel.  
Hence $R(\psi_1) = 0$.
 %\bigskip
%\smallskip
\medskip

 (2) We compute $R(\psi_4)$. This is equal to 
\begin{align}
\tag{*}
\label{eq:*} 
  R\big((e_4\otimes e_5)\eta &+ (e_4\otimes e_6)\omega
     - \gamma(e_1\otimes e_4) - \beta(e_2\otimes e_4)\big)\\
\notag
 &= \pi(\gamma\delta - \beta\varepsilon 
    - \lambda X_4^{m-1}\beta\varepsilon)\eta 
    + \pi(\beta\rho  - \gamma\nu)\omega\\
\notag
 &\ \ \ -\gamma \big(\pi(\delta\eta - \nu\omega)\big)
   - \beta\big(\pi(\rho\omega - \varepsilon\eta 
   - \lambda X_2^{m-1}\varepsilon\eta)\big) 
   .
\end{align}
We must choose a version of $X_4$ and of $X_2$. 
It is natural to take $X_4 = \beta\rho\omega$ and
$X_2 = \rho\omega\beta$. 
  We continue the calculation. 
With this,  
(\ref{eq:*})
is equal to
 \begin{align*}
 & (e_4\otimes \delta\eta + \gamma\otimes \eta - e_4\otimes \varepsilon\eta - \beta\otimes \eta) - \lambda \pi\big((\beta\rho\omega)^{m-1}\beta\varepsilon\big)\eta\cr
 &\ + \ (e_4\otimes \rho\omega + \beta\otimes \omega - e_4\otimes \nu\omega - \gamma\otimes \omega)\cr
 & \ - \  (\gamma\otimes \eta + \gamma\delta\otimes e_4 - \gamma\otimes \omega - \gamma\nu\otimes e_4)\cr
 & - (\beta\otimes \omega + \beta\rho \otimes e_4 - \beta\otimes \eta - \beta\varepsilon\otimes e_4) + \lambda\beta\Big(\pi\big((\rho\omega\beta)^{m-1}\varepsilon\eta\big)\Big) .
 \end{align*}
The terms of the form $\alpha_1\otimes \alpha_2$ 
with $\alpha_i$ arrows all cancel. 
Using the relations
$\delta\eta = \nu\omega$ and $\beta\rho = \gamma\nu$, 
four of the other terms cancel. 
This leaves 
\[ 
  -e_4\otimes \varepsilon\eta + e_4\otimes \rho\omega 
  - \lambda \pi  \big((\beta\rho\omega)^{m-1}\beta\varepsilon\big)\eta
  - \gamma\delta\otimes e_4  + \beta\varepsilon\otimes e_4 
  + \lambda \beta \pi\big((\rho\omega\beta)^{m-1}\varepsilon\eta\big).
\]

The first two terms combine, and the fourth and fifth term 
combine, and we can rewrite the expression as
\begin{align}
 \tag{**}
 \label{eq:**}
  \lambda (e_4\otimes X_2^{m-1}\varepsilon\eta) 
  - \lambda \pi\big((\beta\rho\omega)^{m-1}\beta\varepsilon\big)\eta
  - \lambda (X_4^{m-1}\beta\varepsilon \otimes e_4) 
  + \lambda \beta \pi\big((\rho\omega\beta)^{m-1}\varepsilon\eta\big).
\end{align}
Now we combine the second and fourth term of (\ref{eq:**}), 
and we  expand both. 
All terms except the ones $-\otimes e_4$ and $e_4\otimes -$ 
cancel, and we are left with
\begin{align}
 \tag{***}
 \label{eq:***}
  \lambda\big((\beta\rho\omega)^{m-1}\beta\varepsilon
   \otimes e_4 - e_4
   \otimes (\rho\omega\beta)^{m-1}\varepsilon\eta \big)
   .
\end{align}
The first term of (\ref{eq:***}) is the negative of  
the third term in (\ref{eq:**}) since  $\beta\rho\omega=X_4$. 
The second term of (\ref{eq:***})
is the negative of the first term of (\ref{eq:**}) 
since $\rho\omega\beta = X_2$. 
Hence, everything cancels and 
$R(\psi_4) = 0$, as required.
\end{proof}

\begin{theorem} 
\label{th:6.4} 
There is an isomorphism 
$\Omega^4_{\Lambda^e}(\Lambda) \cong  \Lambda$ 
in $\mod \Lambda^e$. 
\end{theorem}

\begin{proof}
This is similar as in the proof of Theorem 7.4 in \cite{ESk-WSA}. 
We have defined a symmetrizing bilinear form of $\Lambda$
in the proof of Theorem~\ref{th:4.7}. 
We define elements $\xi_i\in \bP_3$  by 
\[
  \xi_i = \sum_{b\in \cB_i} b\otimes b^* ,
\]
where $\{ b^*: b\in \cB\}$ is the dual basis corresponding 
to $\cB$, defined by $(-,-)$.  
As in \cite{ESk-WSA}, it follows that 
the map 
\[
  \theta: \Lambda \to \bP_3, 
  \ \mbox{with} \  \theta(e_i) = \xi_i
  \ \mbox{for all $i \in Q_0$},
\]
is a monomorphism of $\Lambda$-$\Lambda$-bimodules. 
Moreover, one shows that $S(\xi_i)=0$, exactly as in \cite{ESk-WSA}. 
This only uses general properties
of the dual basis and no details on a specific  algebra. 
Furthermore, $\Omega^4_{\Lambda^e}(\Lambda)$ is free 
of rank $1$ as a left or right $\Lambda$-module. 
Namely, we have the exact sequence of bimodules 
\[
  0 \to \Omega^4_{\Lambda^e}(\Lambda) 
  \to \bP_3 \to \bP_2 \to \bP_1 \to \bP_0 \to \Lambda\to 0.
\]
We have  $\bP_0\cong \bP_3$, and moreover $\bP_1$ and $\bP_2$ 
have obviously the same rank as free $\Lambda$-modules on each side.   
By the exactness, it follows
that $\Lambda$ and $\Omega_{\Lambda^e}^4(\Lambda)$ 
have the same rank. 
Therefore, the map $\theta$ gives an isomorphism of $\Lambda$ 
with $\Omega^4_{\Lambda^e}(\Lambda)$. 

Alternatively, 
for the last step one may apply  \cite{GSS} to show  
that $\Omega_{\Lambda^e}^4(\Lambda)$ must be isomorphic 
to $_1\Lambda_{\sigma}$
for some algebra automorphism $\sigma$, and therefore 
has rank $1$ on each side.
\end{proof}

Theorem~\ref{th:main3} follows from
Proposition~\ref{prop:6.1},
Theorem~\ref{th:6.4},
and the following proposition.

\begin{proposition}
\label{prop:6.5} 
Let $A = \Lambda(m,0)$. 
Then $\mod A$ does not admit a periodic simple module.
\end{proposition}

\begin{proof}
Take $i \in \{1,3,5\}$.
Observe that, for the indecomposable projective $A$-modules
$P_i = e_i A$ and $P_{i+1} = e_{i+1} A$, we have
$\rad P_i / \soc P_i \cong \rad P_{i+1} / \soc P_{i+1}$ 
in $\mod A$.
Then, by general theory, $P_i / \soc P_i$
and $P_{i+1} / \soc P_{i+1}$ 
are not in stable tubes of the stable Auslander-Reiten quiver
$\Gamma_A^s$ of $A$.
Since $A$ is a symmetric algebra, we conclude that $S_i$ 
and $S_{i+1}$ are  not periodic modules.
\end{proof}

\section*{Acknowledgements}

The research was done during the visit of
the first named author at the Faculty of Mathematics and Computer Sciences
in Toru\'n (June 2017).

\end{document}